\documentclass[11pt]{article}
\usepackage{indentfirst}
\usepackage{latexsym}
\usepackage{hyperref}
\usepackage{graphicx}
\usepackage{mathrsfs}
\usepackage{bookmark}
\usepackage{enumerate}
\usepackage{amsfonts,amsthm,amssymb,mathtools}
\usepackage{amsmath}
\usepackage{latexsym, euscript, epic, eepic, color}
\usepackage{amssymb}
\usepackage{subfigure}
\usepackage{enumerate}
\usepackage[all]{xy}
\usepackage{setspace}
\allowdisplaybreaks

\usepackage{epstopdf}
\numberwithin{equation}{section}
\newtheorem{theorem}{Theorem}[section]

\newtheorem{definition}[theorem]{Definition}
\newtheorem{conjecture}[theorem]{Conjecture}

\newtheorem{lemma}[theorem]{Lemma}

\newtheorem{example}[theorem]{Example}
\newtheorem{fact}[theorem]{Fact}

 \allowdisplaybreaks

\begin{document}
\baselineskip=16pt

\title{ The minimum spectral radius of graphs with a given domination number \footnote{Corresponding author: Jianping Li (lijianping65@nudt.edu.cn)}
}

\author{ Chang Liu, Jianping Li\\
		\small  College of Sciences, National University of Defense Technology, \\
			\small  Changsha, China, 410073.\\}


\date{\today}

\maketitle

\begin{abstract}
	Let $\mathbb{G}_{n,\gamma}$ be the set of simple and connected graphs on $n$ vertices and with domination number $\gamma$. The graph with minimum spectral radius among $\mathbb{G}_{n,\gamma}$ is called the minimizer graph. In this paper, we first prove that the minimizer graph of $\mathbb{G}_{n,\gamma}$ must be a tree. Moreover, for $\gamma\in\{1,2,3,\lceil\frac{n}{3}\rceil,\lfloor\frac{n}{2}\rfloor\}$, we characterize all minimizer graphs in $\mathbb{G}_{n,\gamma}$.
\end{abstract}



\section{Introduction}
Throughout this paper, we only consider simple, finite, undirected and connected graphs. Let $G$ be a graph with vertex set $V(G)=\{v_1,v_2,\dots,v_n\}$ and edge set $E(G)$. A vertex set $D(G)$ (or $D$ for short) of a graph $G$ is said to a dominating set if every vertex in $V(G)\backslash D(G)$ is adjacent to a vertex in $D(G)$. The domination number of $G$, denoted by $\gamma(G)$ (or $\gamma$ for short), is the minimum cardinality of a dominating set, i.e., $\gamma(G)=\min_{D(G)\subseteq V(G)}\{|D(G)|\}$. A dominating set of $G$ of minimum cardinality is called a $\gamma(G)$-set. 

For a vertex $v\in V(G)$, the neighborhood of $v$ is the set $N_{G}(v)=\{u\in V(G)~|~uv\in E(G)\}$. The degree of $v\in V(G)$, denoted by $d_{G}(v)$, is defined as the number of neighbors of $v$ in $G$, i.e. $d_{G}(v)=|N_{G}(v)|$. A pendent vertex is a vertex of degree one. We say a path $P_{t}=u_1u_2\dots u_t$ ($t\geq 2$) is a pendent path if $u_1$ is a cut vertex; called $u_1$ the root of $P_t$. A vertex $v$ in a tree $T$ is called a support vertex if $v$ is a neighbor of a pendent vertex, while it is called a branching vertex if $d_{T}(G)\geq 3$. The distance between two vertex $u,v\in V(G)$, written $d_{G}(u,v)$, is the length of a shortest path between $u$ and $v$. For a set $S\subseteq V(G)$, we use $G-S$ to denote the graph obtained from $G$ by deleting all vertices in $S$ together with their incident edges.

The adjacency matrix of $G$ is $A(G)=(a_{ij})_{n\times n}$, where $a_{ij}=1$ if and only if $v_i v_j\in E(G)$ and $a_{ij}=0$ otherwise. The characteristic polynomial $\Phi(G,\lambda)=\det\left( \lambda I_n-A(G)\right) $ of the adjacency matrix $A(G)$ of $G$ is called the characteristic polynomial of $G$. Let $\lambda_1\geq\lambda_2\geq\dots\geq\lambda_n$ be eigenvalues of $A(G)$. The spectral radius $\rho(G)$ of $G$ is the largest eigenvalue of $A(G)$, i.e. $\rho(G)=\lambda_1$.


For a long time, many scholars have been interested in the relationships between spectral radius and other graph invariants (e.g., independence number \cite{Jiinde,Louinde,Xuinde}, matching number \cite{Mingmatch,Linmatch}, domination number \cite{Stevdomin,Guandiom}, diameter \cite{Guodia,Vandia}, degree sequence \cite{Liudegree}, the number of cut vertices \cite{Bercut}, the number of pendent vertices \cite{Wupend}, etc).  Particularly, the minimum spectral radius problem is in close relationship with Tur\'an-type problem. This prompted the authors of \cite{Louinde} and \cite{Xuinde} to look for the
graphs having the minimal spectral radius in the class of connected graphs with fixed independence number. In 2011,  Guan and Ye \cite{Guandiom} obtained the unique graph with the minimum spectral radius in the class of $n$-vertex graphs with domination number 2. Let $\mathbb{G}_{n,\gamma}$ be the set of simple and connected graphs with $n$ vertices and domination number $\gamma$. Throughout this paper, we say a graph is a minimizer graph if it attains the minimum spectral radius among $\mathbb{G}_{n,\gamma}$. In this paper, we first proved that the minimizer graph of $\mathbb{G}_{n,\gamma}$ must be a tree. Moreover, for $\gamma\in\{1,2,3,\lceil\frac{n}{3}\rceil,\lfloor\frac{n}{2}\rfloor\}$, we characterize all minimizer graphs in $\mathbb{G}_{n,\gamma}$.

\section{Preliminaries}
In this section, we list some known results which will be used in this paper.

\begin{lemma}\cite{Danlem1,Delalem1}\label{lemdomin}
	Let $D$ be a $\gamma(G)$-set of $G$. Then there exists a spanning tree $T$ of $G$ such that $D$ is a $\gamma(T)$-set. 
\end{lemma}

\begin{lemma}\cite{CveSpec}\label{lemsub}
	If $H$ is a subgraph of $G$, then $\rho(H)\leq\rho(G)$. In particular, if $H$ is proper, then $\rho(H)<\rho(G)$.
\end{lemma}

The following theorem can be easily proven by combining Lemma \ref{lemdomin} and Lemma \ref{lemsub}.
\begin{theorem}\label{themtree}
	Let $G$ be the minimizer graph over all graphs in $\mathbb{G}_{n,\gamma}$. Then $G$ must be a tree.
\end{theorem}

\begin{lemma}\cite{CveSpec}\label{lemmintree}
	The path $P_n$ attains the minimum spectral radius among all connected $n$-vertex graphs and $\rho(P_n)=2\cos\left( \frac{\pi}{n+1}\right) $.
\end{lemma}

\begin{lemma}\cite{BroSpec}\label{lemcalcu}
	Let $T$ be a tree, and $u,v\in V(G)$.
	\begin{itemize}
		\item[\textnormal{(\romannumeral1)}] Let $e=uv$ be an edge in $T$ that separates $T$ into two subtrees $A$ and $B$, with $u\in A$ and $v\in B$. Then $\Phi(T,x)=\Phi(A,x)\Phi(B,x)-\Phi(A-\{u\},x)\Phi(B-\{v\},x)$.
		\item[\textnormal{(\romannumeral2)}] Let $v$ be a vertex of $T$. Then
		$\Phi(T,x)=x\Phi(T-\{v\},x)-\sum_{uv\in E(T)}\Phi(T-\{u,v\},x)$.
	\end{itemize}
\end{lemma}

\begin{lemma}\cite{Lieig}\label{lemlong}
	Let $v\in V(G)$ and suppose that two new path $P_{s+1}^{(1)}=vv_1v_2\dots v_s$ and $P_{t+1}^{(2)}=vu_1u_2\dots u_t$ of length $s,t$ $(1\leq s\leq t)$ are attached to $G$ at $v$, respectively, to form a new graph $G_{s,t}$. Then $\rho(G_{s,t})>\rho(G_{s-1,t+1})$.
\end{lemma}
\begin{figure}[htbp]
	\centering 
	\includegraphics[width=15cm]{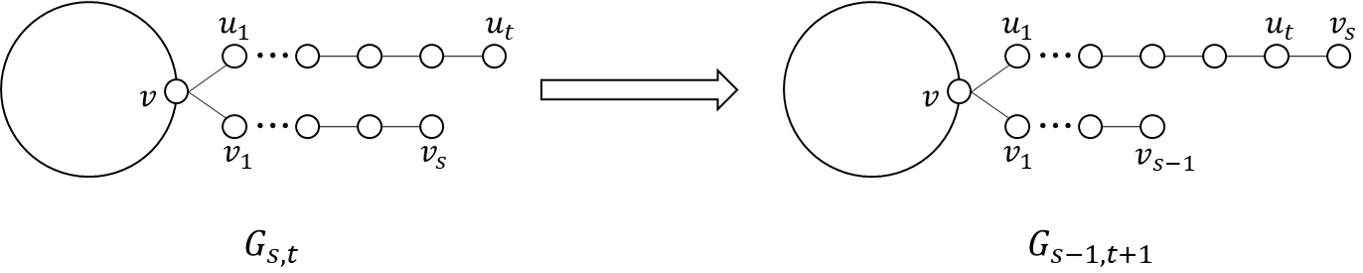}
	\caption{A graph transformation for Lemma \ref{lemlong}.}
\end{figure}

\begin{lemma}\cite{Wupend}\label{lemPero}
	Let $u,v\in V(G)$ be two distinct vertices, $\{w_i|i=1,2,\dots,s\}\subseteq N_{G}(v)\backslash N_{G}(u)$. Suppose that $\boldsymbol{x}=\{x_1,x_2,\dots,x_n\}$ is the Perron vector of $G$, where $x_i$ is corresponding to $w_i$ $(i=1,\dots,n)$. Let $G^*=G-vw_1-vw_2-\dots-vw_s+uw_1+uw_2+\dots+uw_s$. If $x_u\geq x_v$, then $\rho(G^*)>\rho(G)$.
\end{lemma}

\begin{definition}
	An internal path of $G$ is a sequence of vertices $P=v_1v_2\dots v_t$ with $t\geq2$ such that:
		\begin{itemize}
		\item[\textnormal{(\romannumeral1)}] the vertices in the sequence are distinct $(\mbox{except possibly } v_1=v_t)$;			\item[\textnormal{(\romannumeral2)}] $v_i$ is adjacent to $v_{i+1}$ $(i=1,2,\dots,t-1)$;
		\item[\textnormal{(\romannumeral3)}] the two end-vertices of the path have degree strictly greater than 2, and the rest of the vertices with degree 2, i.e. $d_{G}(v_1)\geq3$, $d_{G}(v_t)\geq3$, $d_{G}(v_2)=d_{G}(v_3)=\dots=d_{G}(v_{t-1})=2$ $(\mbox{unless }t=2)$.
	\end{itemize}
\end{definition}

Let $IP(G)$ be the set of the internal paths of $G$. A once (twice) subdivision of a graph $G$, denoted by $G_{u_1}$ ($G_{u_1,u_2}$), is obtained from $G$ by inserting a vertex $u_1$ (two vertices $u_1,u_2$) of degree $2$, to an edge of $G$. 
In this paper, we define a once (twice) subdividing transformation graph of a graph $G$, denoted by $\dot{G}$ ($\ddot{G}$), is obtained from $G$ by once (twice) subdividing an internal edge of $G$, and simultaneously deleting another one vertex (two vertices) on some pendent paths of $G$ such that the order is left unchanged.
 
\begin{lemma}\cite{Hoff}\label{lemsubdiv}
	Suppose that $G\ncong W_n$ and $e=uv$ is an edge on an internal path of $G$. Let $G_w$ $(G_{w_1,w_2})$ be the graph obtained from $G$ by the once $(\mbox{twice})$ subdivision of the edge $uv$. Then $\rho(G_{w})<\rho(G)$ $(\rho(G_{w_1,w_2})<\rho(G))$.
\end{lemma}

\begin{figure}[htbp]
	\centering 
	\includegraphics[width=6cm]{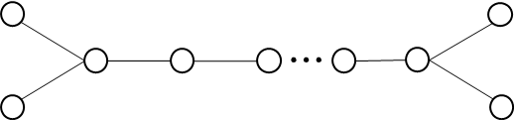}
	\caption{The graph $W_n$.}
\end{figure}

\begin{lemma}\cite{Linorder}\label{lemsymm}
	Let $G$ be a connected graph and $u,v\in V(G)$, where $G-u\cong G-v$. Suppose that $a$ new edges $\{uw_1,uw_2,\dots,uw_a\}$ are attached to $G$ at $u$, and $b$ new pendent edges $\{vw_{a+1},vw_{a+2},\dots,vw_{a+b}\}$ are attached to $G$ at $v$ to form a new graph $G_{a,b}$ $(a\geq b\geq 1)$. Then $\rho(G_{a+1,b-1})>\rho(G_{a,b})$. Specially, $\rho(G_{a+b,0})>\rho(G_{a+b-1,1})>\dots>\rho\left( G_{\lfloor\frac{a+b}{2}\rfloor,\lceil\frac{a+b}{2}\rceil}\right) $.
\end{lemma}

\begin{figure}[htbp]
	\centering 
	\includegraphics[width=16.5cm]{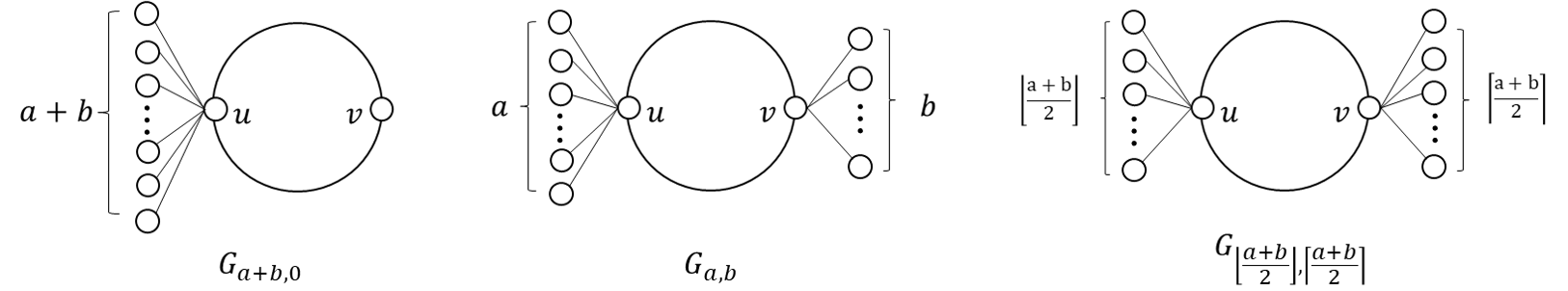}
	\caption{Three graphs for Lemma \ref{lemsymm}.}
\end{figure}

\begin{fact}\label{factset}
	For a tree $T$, there exists a minimum dominating set $D$ of $T$ such that all support vertices of $T$ are in $D$.
\end{fact}
\begin{proof}
	Let $D$ be a minimum dominating set of $T$. Suppose that $u$ is a unique support vertex of $T$ satisfying $u\in V(G)\backslash D$ and $\{v_1,v_2,\dots,v_k\}\subseteq N_{T}(u)$ are pendent vertices, where $k\geq1$. Hence, $\{v_1,v_2,\dots,v_k\}\subseteq D$. Then we have $D'=\left( D\backslash \{v_1,v_2,\dots,v_k\}\right) \cup\{u\}$ is a dominating set of $T$. If $k>1$, it follows that $|D'|<|D|$, a contradiction. If $k=1$, we obtain a minimum dominating set $D'$ of $T$ such that all support vertex of $T$ are in $D'$.
\end{proof}

\begin{fact}\label{factdistance}
	Let $D$ be a minimum dominating set of $G$, then $\min\limits_{\{u,v\}\subseteq D}d_{G}(u,v)\leq3$.
\end{fact}
\begin{proof}
	Suppose that $\min\limits_{\{u,v\}\subseteq D}d_{G}(u,v)=t\geq4$. There exists a path $P_{t+1}=uw_1w_2\dots w_{t-1}v$ between $u$ and $v$ in $G$ such that $\{u,v\}\subseteq D$ and $\{w_1,w_2,\dots,w_{t-1}\}\subseteq V(G)\backslash D$. Thus, we have $N_{G}(w_i)\cap D=\varnothing$, for $i=2,\dots,t-2$, a contradiction.
\end{proof}
	
	Let $D$ be a minimum dominating set of $G$.  For a vertex $u\in D$ , we use $N_{\overline{D}}(u)$ to denote the neighbors of $u$ in non-dominating set $V(G)\backslash D$.
	
\begin{fact}\label{factdomin}
	Let $D=\{v_1,v_2,\dots,v_{\gamma}\}$ be a minimum dominating set of $G$. If $1\leq\gamma<\lceil\frac{n}{3}\rceil$, there must be a vertex $v\in D$ such that $|N_{\overline{D}}(v)|\geq 3$.
\end{fact}
\begin{proof}
	Assume the contrary that $|N_{\overline{D}}(v_i)|\leq2$ for $i=1,2,\dots,\gamma$. It can be checked that $|V(G)\backslash D|=|\bigcup_{i=1}^{\gamma}N_{\overline{D}}(v_i)|\leq \sum_{i=1}^{\gamma}|N_{\overline{D}}(v_i)|\leq 2\gamma$. $|V(G)|=|V(G)\backslash D|+|D|=3\gamma\leq 3\lceil\frac{n}{3}\rceil-3<n$, a contradiction.
\end{proof}
 
 We now introduce some graph transformations. Let $G$ be a tree and let $D$
 be a minimum dominating set of $G$. For convenience, we denote $B(G)$ the set of branching vertices of a tree $G$. Suppose that $u\in B(G)$ is a branching vertex attached some pendent paths $P_{l_i+1}^{(i)} = uv_{i,1}\dots v_{i,l_i}$ ($1\leq i \leq k$) and $l_1=\max_{1\leq i\leq k}l_i$. For $u\in D$, we define the once $Tr_1$-transformation of $G$, written $G'$, as $G'=G-v_{2,1}v_{2,2}-\dots-v_{k,1}v_{k,2}+v_{1,l_1}v_{2,2}+\dots+v_{1,l_1}v_{k,2}$ (see Figure \ref{graphtrans} (a)). For $u\in V(G)\backslash D$, we define the once $Tr_2$-transformation of $G$, denoted by $G''$, as $G''=G-uv_{2,1}-\dots-uv_{k,1}+v_{1,l_1}v_{2,1}+\dots+v_{1,l_1}v_{k,1}$ (see Figure \ref{graphtrans} (b)).
 
 \begin{figure}[htbp]
 	\centering
 	\subfigure[The once $Tr_1$-transformation]{
 		\includegraphics[width=7.6cm]{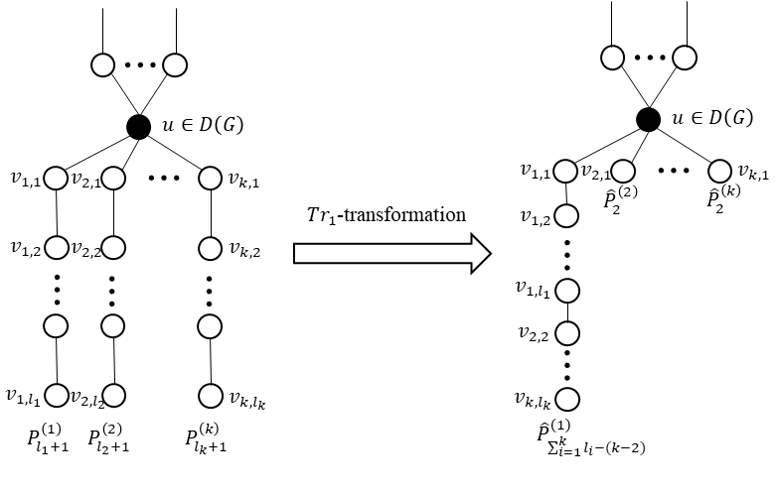}
 	}
 	\quad
 	\subfigure[The once $Tr_2$-transformation]{
 		\includegraphics[width=7.6cm]{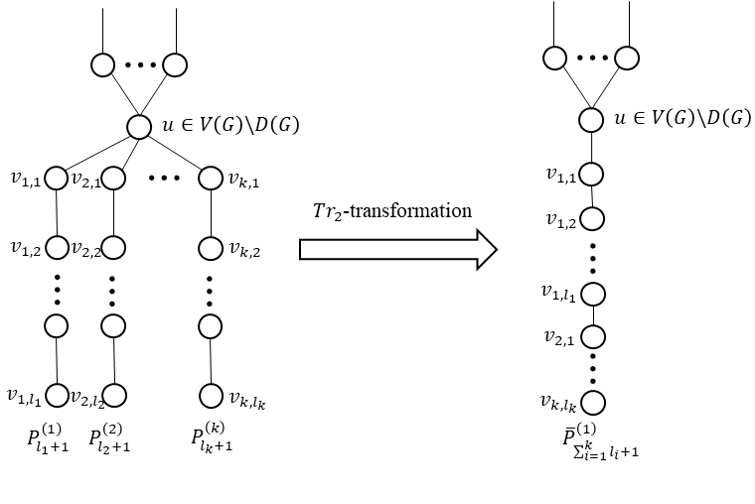}
 	}
 	\caption{Two graph transformations.}
 	\label{graphtrans}
 \end{figure}
 
%
 
	We first repeat the $Tr_2$-transformation until all branching vertices in $V(G)\backslash D(G)$ are traversed, and then we repeat the $Tr_1$-transformation until all branching vertices in $D(G)$ are traversed. Finally, a new tree $G^{\diamond}$ is obtained.  It follows from Lemma \ref{lemlong} that $\rho(G)\geq\rho(G^{\diamond})$ with equality holding if and only if $G\cong G^{\diamond}$. Note that these graph transformations may lead to $\gamma(G^{\diamond})\ne \gamma(G)$.
 
    
     Let $\mathcal{P}_1$ be the set containing all pendent paths of $G^{\diamond}$ whose roots belong to $B(G^{\diamond})\cap D(G)$, and let $\mathcal{P}_2$ be the set containing all pendent paths of $G^{\diamond}$ whose roots belong to $B(G^{\diamond})\cap \left( V(G)\backslash D(G)\right) $ (see Figure \ref{pathset}). Then, we give the following theorem. We use $P_{G^{\diamond}}(u,v)$ to denote the path between $u$ and $v$ of $G^{\diamond}$, where $u,v\in V(G^{\diamond})$. Note that $G^{\diamond}$ is a tree, the path $P_{G^{\diamond}}(u,v)$ is unique.
  
     \begin{figure}[htbp]
    	\centering
    	\subfigure[The pendent path set $\mathcal{P}_1$.]{
    		\includegraphics[width=7.8cm]{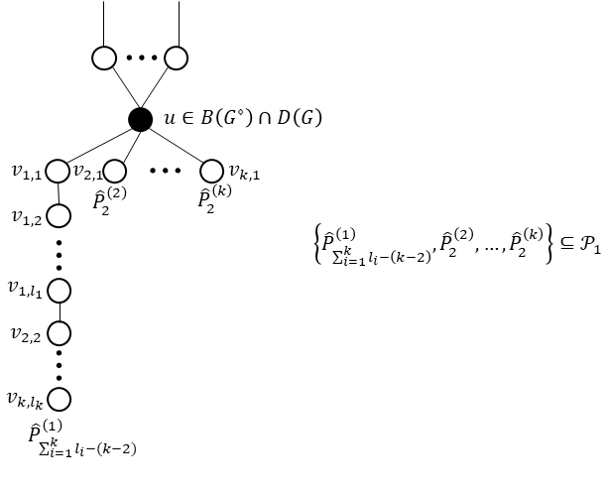}
    	}
    	\quad
    	\subfigure[The pendent path set $\mathcal{P}_2$.]{
    		\includegraphics[width=5cm]{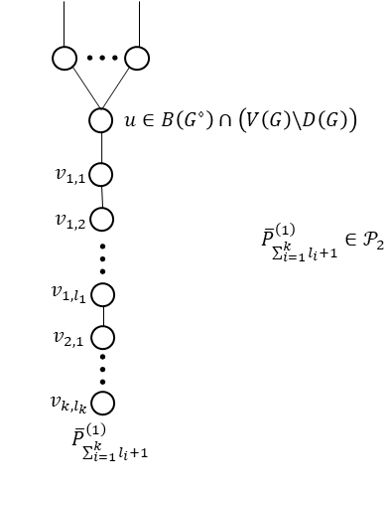}
    	}
    	\caption{Two sets of pendent paths of $G^{\diamond}$.}
    	\label{pathset}
    \end{figure}
%

\begin{theorem}\label{themdistance}
	Let $G$ be the minimizer graph in $\mathbb{G}_{n,\gamma}$ $\left( 1\leq\gamma<\lceil\frac{n}{3}\rceil\right)$ and let $D(G)$ be a minimum dominating set of $G$, then $\min\limits_{\{u,v\}\subseteq D(G)}d_{G}(u,v)=3$. 
\end{theorem}
\begin{proof}
	Let $G$ be the minimizer graph in $\mathbb{G}_{n,\gamma}$, and let $D(G)$ be a minimum dominating set of $G$ containing all support vertices of $G$. It follows from Theorem \ref{themtree} that $G$ is a tree.  Assume to the contrary that $\min_{\{u,v\}\subseteq D(G)}d_{G}(u,v)\ne3$.  By Fact \ref{factdistance}, we have $\min_{\{u,v\}\subseteq D(G)}d_{G}(u,v)=1$ or $\min_{\{u,v\}\subseteq D(G)}d_{G}(u,v)=2$. Through the following steps, we will construct a minimizer graph in $\mathbb{G}_{n,\gamma}$ ($1\leq\gamma<\lceil\frac{n}{3}\rceil$).
	
	\textbf{Step 1.} By applying the above transformations to $G$, we obtain a tree $G^{\diamond}$. Obviously, $\rho(G)\geq\rho(G^{\diamond})$, with equality holding if and only if $G\cong G^{\diamond}$.
	
	\textbf{Step 2.} If there exist two vertices $u,v\in D(G)$ such that $d_{G^{\diamond}}(u,v)=1$ and $e=uv\subseteq \tilde{P}$, where $\tilde{P}$ is some path in $IP(G^{\diamond})$. we apply the twice subdividing transformation to $G^{\diamond}$ such that $d_{\ddot{G^{\diamond}}}(u,v)=3$. Repeat this transformation, a new tree $G^{\diamond}_1$ will be obtained finally. It follows that
	\begin{equation*}
		2\leq \min_{\substack{\{u,v\}\subseteq D(G),\\ P_{G^{\diamond}_1}(u,v)\subseteq \tilde{P}^{(1)}}}d_{G^{\diamond}_1}(u,v)\leq 3,
	\end{equation*}
	where $\tilde{P}^{(1)}$ is some path in $IP(G^{\diamond}_1)$.
	
	If there exists two vertices $u',v'\in D(G)$ such that  $d_{G^{\diamond}_1}(u',v')=2$ and $P_{G^{\diamond}_1}(u',v')\subseteq \tilde{P}^{(1)}\in IP(G^{\diamond}_1)$, we utilize the once subdividing transformation on $G^{\diamond}_1$ such that $d_{\dot{G^{\diamond}_1}}(u',v')=3$. Repeat this transformation, finally, we constructed a new tree $G^{\diamond}_2$. It can be checked that 
	\begin{equation*}
		\min_{\substack{\{u,v\}\subseteq D(G) \\ P_{G^{\diamond}_2}(u,v)\subseteq \tilde{P}^{(2)}}}d_{G^{\diamond}_2}(u,v)= 3,
	\end{equation*}
	where  $\tilde{P}^{(2)}$ is some path in $IP(G^{\diamond}_2)$.
	
	When applying the subdividing transformation, we choose the pendent path in the following sequence and remove its vertices:
	
	(a) the pendent path in $\mathcal{P}_2$,
	
	(b) the pendent path in $\mathcal{P}_1$ with length more than 1,
	
	(c) the pendent path in $\mathcal{P}_1$ with length $1$.
	
	\begin{figure}[htbp]
		\centering 
		\includegraphics[width=10cm]{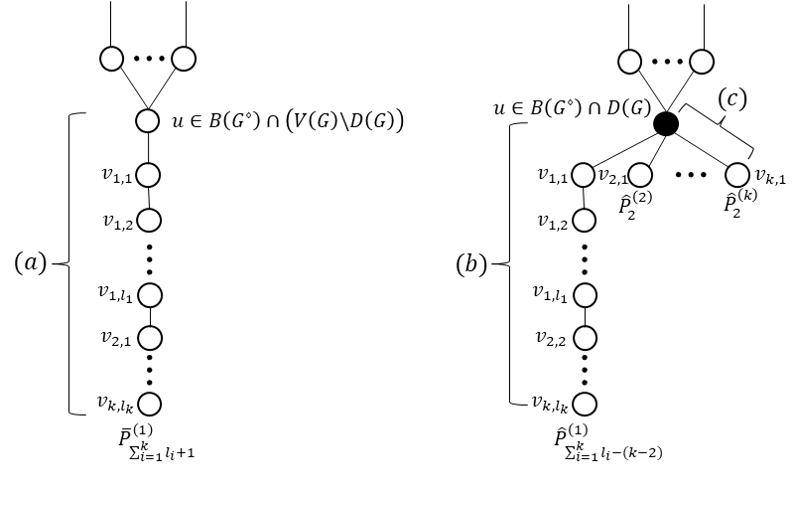}
		\caption{The sequence of choosing pendent paths.}
	\end{figure}
	
	 From Lemma \ref{lemsubdiv}, it's easy to see that $\rho(G^{\diamond})\geq\rho(G^{\diamond}_2)$ with equality holding if and only if $G^{\diamond}\cong G^{\diamond}_2$.
	
	\textbf{Step 3.} From Fact \ref{factdomin}, one can see that even though $\min_{\{u,v\}\subseteq D(G)}d_{G}(u,v)=3$ in a tree $G\in \mathbb{G}_{n,\gamma}$, there will still be at least one vertex $v\in D(G)$, such that $d_{G}(v)=|N_{\overline{D(G)}}(v)|\geq 3$, implying that $v$ is a branching vertex of $G$.
	
	If $G^{\diamond}_2$ has exactly one branching vertex, we denote $v$ the only branching vertex of $G^{\diamond}_2$. Pick the longest pendent path $P'$ of $G^{\diamond}_2$ attached to $v$, and move the terminal vertex from other pendent paths attached to $v$ to $P'$ repeatedly. Until the tree $G^{*}_1$ and the longest pendent path $P''$ attached to $v$ in $G^{*}_1$ are constructed, satisfying all vertices in $N_{G^{*}_1}(u)-V(P'')$ are pendent vertices, $d_{G^{*}_1}(v)=|N_{\overline{D(G^{*}_1)}}(v)|=n-3(\gamma-1)-1$ and $|V(P'')|=3(\gamma-1)+2$.
	
	Combining Fact \ref{factset}, we get that there exists a minimum dominating set $D(G^{*}_1)$ of $G^{*}_1$ satisfying $|D(G^{*}_1)|=\gamma$ and $\min_{\{u,v\}\subseteq D(G^{*}_1)}=3$.  By Lemma \ref{lemlong}, we have $\rho(G^{\diamond}_2)>\rho(G^{*}_1)$ with equality holding if and only if $G^{\diamond}_2\cong G^{*}_1$.

	If $G^{\diamond}_2$ has at least two branching vertices, we pick an internal path $P_{m}=w_1w_2\dots w_m$ ($m\geq2$) of $G^{\diamond}_2$, and apply the subdividing transformation to an edge of the internal path between $w_1$ and $w_m$ repeatedly (Select the pendent paths as the sequence shown in Step 2 and remove their vertices). Let $\{u_1,\dots,u_s\}$ be roots of remaining pendent paths in $\mathcal{P}_2$. Finally, a new tree $G^{*}_2$ is obtained, satisfying $\mathcal{P}_2=\varnothing$, all pendent paths in $\mathcal{P}_1$ have length 1, and $\sum_{i=1}^{s}d_{G^{*}_2}(u_i)= \sum_{i=1}^{s}|N_{\overline{D(G^{*}_2)}}(u_i)|=n-3(\gamma-s)-s$.
	\begin{figure}[htbp]
		\centering 
		\includegraphics[width=15cm]{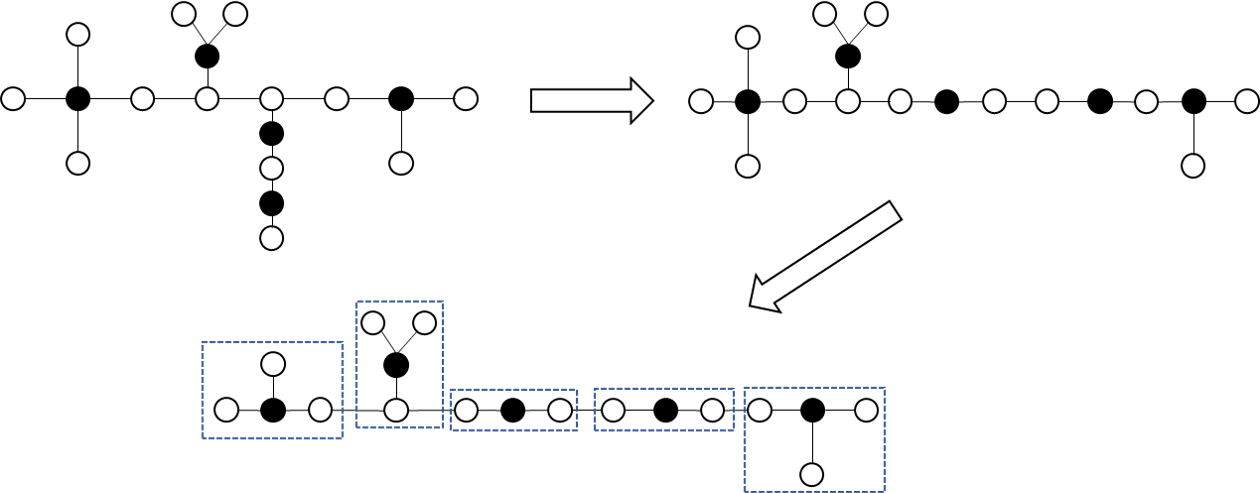}
		\caption{Three trees with 18 vertices and domination number 5.}
	\end{figure}

	By Fact \ref{factset}, one can see that there exists a minimum dominating set $D(G^{*}_2)$ of $G^{*}_2$ satisfying $|D(G^{*}_2)|=\gamma$ and $\min_{\{u,v\}\subseteq D(G^{*}_2)}=3$. Combining Lemma \ref{lemlong} and Lemma \ref{lemsubdiv}, we have $\rho(G^{\diamond}_2)\geq\rho(G^{*}_2)$ with equality holding if and only if $G^{\diamond}_2\cong G^{*}_2$.
	
	Notice that the above equalities do not hold at the same time. Given the above, for a tree $G\in \mathbb{G}_{n,\gamma}$ with $1\leq\gamma\leq \lceil\frac{n}{3}\rceil$, if $\min_{\{u,v\}\subseteq D(G)}d_{G}(u,v)\ne 3$, we can always construct a new tree $G^{*}$ based on $G$ such that $G^{*}\in\mathbb{G}_{n,\gamma}$, $\min_{\{u,v\}\subseteq D(G^{*})}d_{G^{*}}(u,v)= 3$ and $\rho(G)>\rho(G^{*})$, a contradiction.
	
	This completes the proof.
\end{proof}

\section{The minimizer graphs in $\mathbb{G}_{n,\gamma}$ for $\gamma\in\{1,2,\lceil\frac{n}{3}\rceil\}$}

 We use $S_n$ and $P_n$ to denote the $n$-vertex star and the $n$-vertex path, respectively. Let $H_{\lfloor\frac{n}{2}\rfloor,\lceil\frac{n}{2}\rceil}$ be a graph obtained from a path $v_1v_2v_3v_4$ by attaching $\lfloor\frac{n}{2}\rfloor-2$ pendent edges to $v_1$ and attaching $\lceil\frac{n}{2}\rceil-2$ pendent edges to $v_4$.
 
 	\begin{figure}[htbp]
 		\centering 
 		\includegraphics[width=8cm]{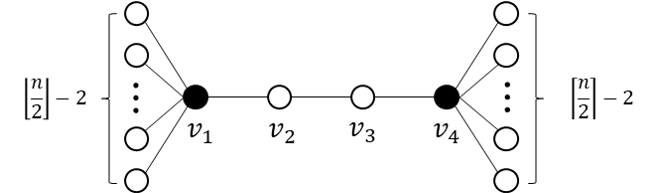}
 		\caption{The graph $H_{\lfloor\frac{n}{2}\rfloor,\lceil\frac{n}{2}\rceil}$.}
 	\end{figure}

 In \cite{Guandiom}, Guan and Ye have characterized the minimizer graph $H_{\lfloor\frac{n}{2}\rfloor,\lceil\frac{n}{2}\rceil}$ in $\mathbb{G}_{n,2}$. In this paper, we can directly draw this conclusion based on Lemma \ref{lemsymm} and Theorem \ref{themdistance}.

\begin{theorem}
	Let $G$ be the minimizer graph with domination number $\gamma$ in $\mathbb{G}_{n,}$, then
	\begin{equation*}
		G\cong \begin{cases}
			S_n & \gamma=1,\\
			H_{\lfloor\frac{n}{2}\rfloor,\lceil\frac{n}{2}\rceil} & \gamma=2, \\
			P_n & \gamma = \lceil\frac{n}{3}\rceil.
		\end{cases}
	\end{equation*}
\end{theorem}

\section{The minimizer graph in $\mathbb{G}_{n,3}$}

In this section, we determine graphs with the minimum spectral radius in $\mathbb{G}_{n,3}$. Let $G_1(a_1,b_1,c_1)$ be a graph obtained from a path $v_1v_4v_5v_6v_3$ by attaching $a_1$ pendent edges to $v_1$, attaching $c_1$ pendent edges to $v_3$ and joining $v_5$ with the centre $v_2$ of a $(b_1+1)$-vertex star, where $a_1+b_1+c_1=n-6$.  Let $G_2(a_2,b_2,c_2)$ be a graph obtained from a path $v_1v_4v_5v_2v_6v_7v_3$ by attaching $a_2$ pendent edges to $v_1$, $b_2$ pendent edges to $v_2$ and $c_2$ pendent edges to $v_3$, where $a_2+b_2+c_2=n-7$.
First, we prove several necessary lemmas.
\begin{figure}[htbp]
	\centering 
	\includegraphics[width=12.8cm]{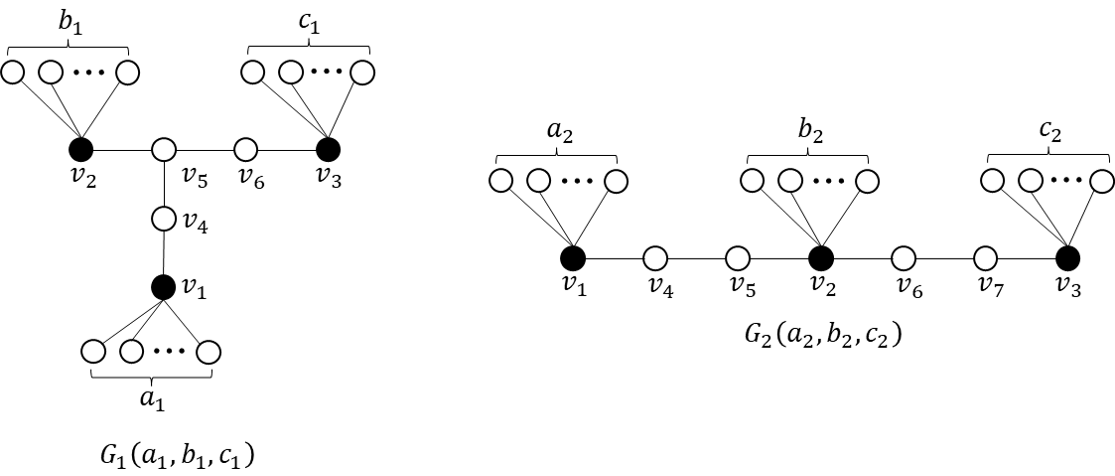}
	\caption{The graphs $G_1(a_1,b_1,c_1)$ and $G_2(a_2,b_2,c_2)$.}
\end{figure}

\begin{lemma}\label{lemform} 
	If $G$ is the minimizer graph in $\mathbb{G}_{n,3}$ $\left( n\geq11\right)$, then $G$ is isomorphic to some graph $G_2(a,b,c)$ with $a+b+c=n-7$.
\end{lemma}
\begin{proof}
	Without loss of generality, we assume $D(G)=\{v_1,v_2,v_3\}$ is a minimum dominating set of $G$. It follows from Theorem \ref{themdistance} that $G\cong G_1(a_1,b_1,c_1)$ or $G\cong G_2(a_2,b_2,c_2)$, where $a_i\geq1$,  $c_i\geq1$, $i=1,2$, $a_1+b_1+c_1=n-6$ and $a_2+b_2+c_2=n-7$.
	
	Let $\boldsymbol{x}=\{x_{v_1},x_{v_2},\dots,x_{v_n}\}$ be the Perron vector of $G_2(a_1,b_1-1,c_1)$. If $x_{v_5}\geq x_{v_6}$, from Lemma \ref{lemPero}, we obtain $\rho\left( G_2(a_1,b_1-1,c_1)\right) <\rho\left( G_2(a_1,b_1-1,c_1)-v_6v_7+v_5v_7\right) $. It can be checked that $G_1(a_1,b_1,c_1)\cong G_2(a_1,b_1-1,c_1)-v_6v_7+v_5v_7$. Hence, $\rho(G_1(a_1,b_1,c_1))>\rho(G_2(a_1,b_1-1,c_1))$. Analogously, if $x_{v_5}<x_{v_6}$, we have $G_1(a_1,b_1,c_1)\cong G_2(a_1,b_1-1,c_1)-v_4v_5+v_4v_6$ and $\rho(G_1(a_1,b_1,c_1))>\rho(G_2(a_1,b_1-1,c_1))$.
	
	As can be seen, if $G\in\mathbb{G}_{n,3}$ ($n\geq11$) have the minimum spectral radius, then $G\cong G_2(a,b,c)$ with $a+b+c=n-7$.
\end{proof}

\begin{lemma}
	Let $G\in \mathbb{G}_{n,3}$ $\left( n\geq11\right)$  has the minimum spectral radius, then $G$ is isomorphic to some graph $G_2(a,n-2a-7,a)$.
\end{lemma}
\begin{proof}
	We assume, without loss of generality, that $a\geq c$.  If $G\in\mathbb{G}_{n,3}$ ($n\geq11$) has the minimum spectral radius, by Lemma \ref{lemform} and Lemma \ref{lemsymm}, we have $G\cong G_2(a,n-2a-7,a)$ or $G\cong G_2(a,n-2a-6,a-1)$. Let $\boldsymbol{x}=\{x_{v_1},x_{v_2},\dots,x_{v_n}\}$ be the Perron vector of $G$. Assume to the contrary that $G\cong G_2(a,n-2a-6,a-1)$. By Lemma \ref{lemcalcu}, we calculate the characteristic polynomials of $G_2(a,n-2a-7,a)$, $G_2(a,n-2a-6,a-1)$, and $G_2(a-1,n-2a-5,a-1)$, which are
	\begin{equation*}
		\begin{split}
			&~~~~\Phi(G_2(a,n-2a-7,a),\lambda)\\
			&=\left[ \Phi(S_{a+2},\lambda)\Phi(S_{n-2a-4},\lambda)-\Phi(S_{a+1},\lambda)\Phi(S_{n-2a-5},\lambda)\right]\Phi(S_{a+2},\lambda)\\
			&~~~~-\left[ \Phi(S_{a+2},\lambda)\Phi(S_{n-2a-5},\lambda)-\Phi(S_{a+1},\lambda)\Phi(S_{n-2a-6},\lambda)\right] \Phi(S_{a+1},\lambda)\\
			&=\left\lbrace \lambda^a[\lambda^2-(a+1)]\cdot \lambda^{n-2a-6}[\lambda^2-(n-2a-5)]-\lambda^{a-1}(\lambda^2-a)\cdot \lambda^{n-2a-7}[\lambda^2-(n-2a-6)]\right\rbrace\\
			&~~~~\cdot \lambda^a(\lambda^2-(a+1))-\left\lbrace \lambda^a[\lambda^2-(a+1)]\cdot \lambda^{n-2a-7}[\lambda^2-(n-2a-6)]-\lambda^{a-1}(\lambda^2-a)\right.\\
			&~~~~\left.\cdot \lambda^{n-2a-8}[\lambda^2-(n-2a-7)]\right\rbrace\cdot \lambda^{a-1}(\lambda^2-a)\\
			&=\lambda^{n-10}\left\lbrace \lambda^{10}+(1-n)\lambda^8+[2n-7a+an-(a+2)(a-n+3)-2a^2-12]\lambda^6+\left[7a-an+a(a-n+3)\right.\right.\\
			&~~~~\left.\left.+(a-2)(8a-2n-an+2a^2+12)+2a^2 \right]\lambda^4+\left[-(a+2)(7a-an+2a^2)\right.\right.\\
			&~~~~\left.\left.-a(8a-2n-an+2a^2+12) \right]\lambda^2
			+a(7a-an+2a^2)   \right\rbrace,
		\end{split}
	\end{equation*}
	\begin{equation*}
		\begin{split}
			\begin{split}
				&~~~~\Phi(G_2(a,n-2a-6,a-1),\lambda)\\
				&=\left[ \Phi(S_{a+2},\lambda)\Phi(S_{n-2a-3},\lambda)-\Phi(S_{a+1},\lambda)\Phi(S_{n-2a-4},\lambda)\right]\Phi(S_{a+1},\lambda)\\
				&~~~~-\left[ \Phi(S_{a+2},\lambda)\Phi(S_{n-2a-4},\lambda)-\Phi(S_{a+1},\lambda)\Phi(S_{n-2a-5},\lambda)\right] \Phi(S_{a},\lambda)\\
				&=\left\lbrace \lambda^a[\lambda^2-(a+1)]\cdot \lambda^{n-2a-5}[\lambda^2-(n-2a-4)]-\lambda^{a-1}(\lambda^2-a)\cdot \lambda^{n-2a-6}[\lambda^2-(n-2a-5)]\right\rbrace\\
				&~~~~\cdot \lambda^{a-1}(\lambda^2-a)-\left\lbrace \lambda^a[\lambda^2-(a+1)]\cdot \lambda^{n-2a-6}[\lambda^2-(n-2a-5)]-\lambda^{a-1}(\lambda^2-a)\right.\\
				&~~~~\left.\cdot \lambda^{n-2a-7}[\lambda^2-(n-2a-6)]\right\rbrace\cdot \lambda^{a-2}(\lambda^2-(a-1))\\
				&=\lambda^{n-10}\left\lbrace \lambda^{10}+(1-n)\lambda^8+(3n-9a+2an-3a^2-13)\lambda^6+(24a-n-5an-a^2n+12a^2+2a^3+8)\lambda^4\right.\\
				&~~~~\left.+(2an-2n-9a+2a^2n-13a^2-4a^3+11)\lambda^2+an-6a-a^2n+4a^2+2a^3   \right\rbrace,
			\end{split}
		\end{split}
	\end{equation*}
	and
	\begin{equation*}
		\begin{split}
			\begin{split}
				&~~~~\Phi(G_2(a-1,n-2a-5,a-1),\lambda)\\
				&=\left[ \Phi(S_{a+1},\lambda)\Phi(S_{n-2a-2},\lambda)-\Phi(S_{a},\lambda)\Phi(S_{n-2a-3},\lambda)\right]\Phi(S_{a+1},\lambda)\\
				&~~~~-\left[ \Phi(S_{a+1},\lambda)\Phi(S_{n-2a-3},\lambda)-\Phi(S_{a},\lambda)\Phi(S_{n-2a-4},\lambda)\right] \Phi(S_{a},\lambda)\\
				&=\left\lbrace \lambda^{a-1}(\lambda^2-a)\cdot \lambda^{n-2a-4}[\lambda^2-(n-2a-3)]-\lambda^{a-2}[\lambda^2-(a-1)]\cdot \lambda^{n-2a-5}[\lambda^2-(n-2a-4)]\right\rbrace\\
				&~~~~\cdot \lambda^{a-1}(\lambda^2-a)-\left\lbrace \lambda^{a-1}(\lambda^2-a)\cdot \lambda^{n-2a-5}[\lambda^2-(n-2a-4)]-\lambda^{a-2}[\lambda^2-(a-1)]\right.\\
				&~~~~\left.\cdot \lambda^{n-2a-6}[\lambda^2-(n-2a-5)]\right\rbrace\cdot \lambda^{a-2}(\lambda^2-(a-1))\\
				&=\lambda^{n-10}\left\lbrace
				\lambda^{10}+(1-n)\lambda^8+[n-3a+an-(a+1)(a-n+2)-2a^2-7]\lambda^6+\left[3a+n-an\right.\right.\\
				&~~~~\left.\left.+(a+1)(4a-n-an+2a^2+6)+(a-1)(a-n+2)+2a^2-5 \right]\lambda^4+\left[-(a+1)(3a+n-an+2a^2-5)\right.\right.\\
				&~~~~\left.\left.-(a-1)(4a-n-an+2a^2+6) \right]\lambda^2
				+(a-1)(3a+n-an+2a^2-5)   \right\rbrace,
			\end{split}
		\end{split}
	\end{equation*}
	respectively.

	Let $\mu=\lambda^2$,
	\begin{equation*}
		\begin{split}
			f_1(\mu,n)&=[\mu^2-(2+a)\mu+a]\cdot[\mu^3+(a-n+3)\mu^2+(2n-8a+an-2a^2-12)\mu+7a-an+2a^2],\\
			f_2(\mu,n)&=\mu^5+(1-n)\mu^4+(3n-9a+2an-3a^2-13)\mu^3+(24a-n-5an-a^2n+12a^2+2a^3+8)\mu^2\\
			&~~~~+(2an-2n-9a+2a^2n-13a^2-4a^3+11)\mu+an-6a-a^2n+4a^2+2a^3,
		\end{split}
	\end{equation*}
	and
	\begin{equation*}
		f_3(\mu,n)=[\mu^2-(a+1)\mu+(a-1)]\cdot[\mu^3+(a-n+2)\mu^2+(n-4a+an-2a^2-6)\mu+(a-1)(2a-n+5)].
	\end{equation*}
	
	Hence, 
	\begin{equation*}
		\begin{split}
			\Phi(G_2(a,n-2a-7,a),\lambda)&=\mu^{\frac{n-10}{2}}f_1(\mu,n),\\
			\Phi(G_2(a,n-2a-6,a-1),\lambda)&=\mu^{\frac{n-10}{2}}f_2(\mu,n),\\
			\Phi(G_2(a-1,n-2a-5,a-1),\lambda)&=\mu^{\frac{n-10}{2}}f_3(\mu,n).
		\end{split}
	\end{equation*}
	
	The roots of functions $f_i(\mu,n)$ are denoted by $\mu_{i,1}(n)\geq\mu_{i,2}(n)\geq\cdots\geq\mu_{i,5}(n)$ for $i=1,2,3$. One can see that
	\begin{equation*}
		\begin{split}
			\rho^2(G_2(a,n-2a-7,a))&=\mu_{1,1}(n).\\
			\rho^2(G_2(a,n-2a-6,a-1))&=\mu_{2,1}(n),\\
			\rho^2(G_2(a-1,n-2a-5,a-1))&=\mu_{3,1}(n).
		\end{split}
	\end{equation*}
	
	A straightforward calculation shows that $\mu_{1,2}=\frac{a}{2}+\frac{\sqrt{a^2+4}}{2}+1$, $\mu_{1,4}=\frac{a}{2}-\frac{\sqrt{a^2+4}}{2}+1$, $\mu_{3,2}=\frac{a}{2}+\frac{\sqrt{a^2-2a+5}}{2}+\frac{1}{2}$ and $\mu_{3,4}=\frac{a}{2}-\frac{\sqrt{a^2-2a+5}}{2}+\frac{1}{2}$.
	
	Without loss of generality, we assume that
	\begin{equation*}
		\begin{split}
			N_{G_2(a,n-2a-7,a)}(v_1)&=\{v_4,\dots,v_{a+3},v_{a+4}\},\\ N_{G_2(a,n-2a-7,a)}(v_2)&=\{v_{a+5},v_{a+6},\dots,v_{n-a-2},v_{n-a-1}\},\\
			N_{G_2(a,n-2a-7,a)}(v_3)&=\{v_{n-a},v_{n-a+1},\dots,v_n\},\\
			N_{G_2(a-1,n-2a-5,a-1)}(v_1)&=\{v_4,\dots,v_{a+2},v_{a+3}\},\\ N_{G_2(a-1,n-2a-5,a-1)}(v_2)&=\{v_{a+4},v_{a+5},\dots,v_{n-a-1},v_{n-a}\},\\
			N_{G_2(a-1,n-2a-5,a-1)}(v_3)&=\{v_{n-a+1},v_{n-a+2},\dots,v_n\}.
		\end{split}
	\end{equation*}
\begin{figure}[htbp]
	\centering
	\subfigure[$G_2(a,n-2a-7,a)$]{
		\includegraphics[width=7.6cm]{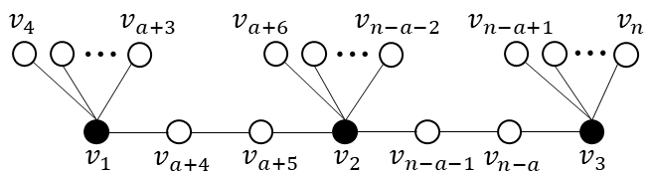}
	}
	\quad
	\subfigure[$G_2(a,n-2a-6,a-1)$]{
		\includegraphics[width=7.6cm]{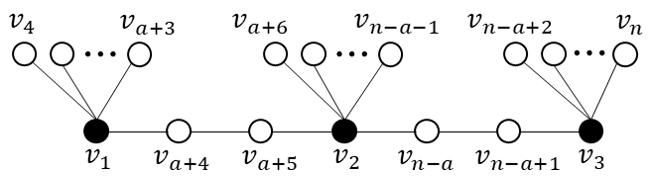}
		
	}
	\quad
	\subfigure[$G_2(a-1,n-2a-5,a-1)$]{
		\includegraphics[width=7.6cm]{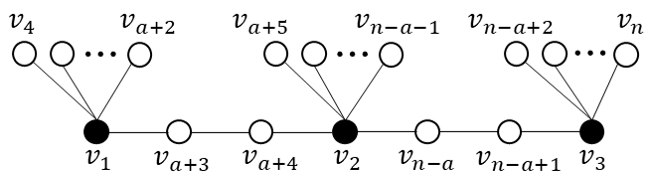}
		
	}
	\caption{The graphs $G_2(a,n-2a-7,a)$, $G_2(a,n-2a-6,a-1)$, and $G_2(a-1,n-2a-5,a-1)$.}
	\label{theregraph}
\end{figure}

	For $G_2(a,n-2a-7,a)$, it is clear that
	\begin{align}
		\rho x_{v_4}=\dots=\rho x_{v_{a+3}}&=x_{v_1},\label{sys1}\\
		\rho x_{v_1}&=ax_{v_4}+x_{v_{a+4}},\label{sys2}\\
		\rho x_{v_{a+4}} &= x_{v_1}+x_{v_{a+5}}, \label{sys3}\\
		\rho x_{v_{a+5}} &= x_{v_{a+4}}+x_{v_2}, \label{sys4}\\
		\rho x_{v_{a+6}} = \dots = \rho x_{v_{n-a-2}} &= x_{v_2}, \label{sys5}\\
		\rho x_{v_2} &= x_{v_{a+5}} + x_{v_{n-a-1}} + (n-2a-7)x_{v_{a+6}}.\label{sys6}
 	\end{align}

 	Combining \eqref{sys1}-\eqref{sys2} and \eqref{sys5}-\eqref{sys6}, we have
 	\begin{equation*}
 		\begin{split}
 			\rho^2 x_{v_4}=\rho x_{v_1}=ax_{v_4}+x_{v_{a+4}}&\Rightarrow x_{v_{a+4}}=(\rho^2-a)x_{v_4},\\
 			\rho^2 x_{v_{a+6}} =\rho x_{v_2} = x_{v_{a+5}} + x_{v_{n-a-1}} + (n-2a-7)x_{v_{a+6}} &\Rightarrow x_{v_{a+5}} = \frac{1}{2}[\rho^2-(n-2a-7)]x_{v_{a+6}}.
 		\end{split}
 	\end{equation*}
 	
 	From \eqref{sys3}-\eqref{sys4}, it follows that 
 	\begin{equation*}
 		\begin{split}
 			&\rho(x_{v_{a+4}}-x_{v_{a+5}})=(x_{v_1}-x_{v_2})+(x_{v_{a+5}}-x_{v_{a+4}})\\
 			&\Leftrightarrow (\rho+1)(x_{v_{a+4}}-x_{v_{a+5}})=x_{v_1}-x_{v_2}\\
 			&\Leftrightarrow(\rho+1)\left\lbrace (\rho^2-a)x_{v_4}-\frac{1}{2}[\rho^2-(n-2a-7)]x_{v_{a+6}}\right\rbrace =x_{v_1}-x_{v_2}\\
 			&\Leftrightarrow \frac{(\rho^2-a)(\rho+1)}{\rho}x_{v_1}-\frac{\frac{1}{2}[\rho^2-(n-2a-7)](\rho+1)}{\rho}x_{v_2}=x_{v_1}-x_{v_2}\\
 			&\Leftrightarrow \frac{(\rho^2-a)(\rho+1)-\rho}{\rho}x_{v_1}=\frac{\frac{1}{2}[\rho^2-(n-2a-7)](\rho+1)-\rho}{\rho}x_{v_2}\\
 			&\Leftrightarrow \frac{x_{v_1}}{x_{v_2}}=\frac{\frac{1}{2}[\rho^2-(n-2a-7)](\rho+1)-\rho}{(\rho^2-a)(\rho+1)-\rho}.
 		\end{split}
 	\end{equation*}
 
 	Note that $G_2(a,n-2a-7,a)$ contains $S_{a+2}$ as a proper subgraph. By Lemma \ref{lemsub}, we get $\rho^2(G_2(a,n-2a-7,a))>a+1$. Thus,  $(\rho^2-a)(\rho+1)-\rho=[\rho^2-(a+1)]\rho+(\rho^2-a)>0$. If $\frac{x_{v_1}}{x_{v_2}}<1$, one can see that
 	\begin{align}
 			&\frac{x_{v_1}}{x_{v_2}}=\frac{\frac{1}{2}[\rho^2-(n-2a-7)](\rho+1)-\rho}{(\rho^2-a)(\rho+1)-\rho}<1\label{inequ1}\\
 			&\Leftrightarrow \frac{1}{2}[\rho^2-(n-2a-7)](\rho+1)-\rho<(\rho^2-a)(\rho+1)-\rho\nonumber\\
 			&\Leftrightarrow\frac{1}{2}(\rho+1)[(4a-n+7)-\rho^2]<0\nonumber.
 	\end{align}
 
 	Since $\rho^2(G_2(a,n-2a-7,a))>a+1$, if $(4a-n+7)\leq a+1$, i.e., $n\geq 3a+6$, we have $[(4a-n+7)-\rho^2]<0$. Then inequality \eqref{inequ1} holds, implying $x_{v_1}<x_{v_2}$. Notice that $G_2(a,n-2a-6,a-1)\cong G_2(a,n-2a-7,a)-v_1v_{4}+v_2v_4$. From Lemma \ref{lemPero}, we have $\rho(G_2(a,n-2a-6,a-1))>\rho(G_2(a,n-2a-7,a))$, a contradiction.
 	
 	For $G_2(a-1,n-2a-5,a-1)$, it follows that
 	\begin{align}
 		\rho x_{v_4}=\dots=\rho x_{v_{a+2}}&=x_{v_1},\label{sysm1}\\
 		\rho x_{v_1}&=(a-1)x_{v_4}+x_{v_{a+3}},\label{sysm2}\\
 		\rho x_{v_{a+3}} &= x_{v_1}+x_{v_{a+4}}, \label{sysm3}\\
 		\rho x_{v_{a+4}} &= x_{v_{a+3}}+x_{v_2}, \label{sysm4}\\
 		\rho x_{v_{a+5}} = \dots = \rho x_{v_{n-a-1}} &= x_{v_2}, \label{sysm5}\\
 		\rho x_{v_2} &= x_{v_{a+4}} + x_{v_{n-a}} + (n-2a-5)x_{v_{a+5}}.\label{sysm6}
 	\end{align}
Similarly, we calculate that
\begin{equation*}
	\frac{x_{v_1}}{x_{v_2}}=\frac{\frac{1}{2}[\rho^2-(n-2a-5)](1+\rho)-\rho}{[\rho^2-(a-1)](1+\rho)-\rho}.
\end{equation*}

Since $G_2(a-1,n-2a-5,a-1)$ contains $S_{a+1}$ as a proper subgraph, from Lemma \ref{lemsub}, we have $\rho^2(G_2(a-1,n-2a-5,a-1))>a$. Hence, $[\rho^2-(a-1)](1+\rho)-\rho=(\rho^2-a)\rho+[\rho^2-(a-1)]>0$. If $\frac{x_{v_1}}{x_{v_2}}>1$, we have
\begin{equation*}
	\begin{split}
		&\frac{x_{v_1}}{x_{v_2}}=\frac{\frac{1}{2}[\rho^2-(n-2a-5)](1+\rho)-\rho}{[\rho^2-(a-1)](1+\rho)-\rho}>1\\
		&\Leftrightarrow \frac{1}{2}[\rho^2-(n-2a-5)](1+\rho)-\rho>[\rho^2-(a-1)](1+\rho)-\rho\\
		&\Leftrightarrow \frac{1}{2}(\rho+1)[(4a-n+3)-\rho^2]>0.
	\end{split}
\end{equation*}

By plugging the value $a+2$ into $\mu$ of $f_3(\mu,n)$, we have
\begin{equation*}
	f_3(a+2,n)=-(2a+1)^2n+12a^3+32a^2+11a-1.
\end{equation*}

For $n\leq 3a+1$, one can see that $4a-n+3\geq a+2$ and $f_3(a+2,3a+1)=-(2a+1)^2n+12a^3+32a^2+11a-1\geq16a^2+4a-2>0$ ($a\geq1$). Note that $a+2>\frac{a}{2}+\frac{\sqrt{a^2-2a+5}}{2}+\frac{1}{2}=\mu_{3,2}$. We have $a+2\geq\mu_{3,1}=\rho^2(G_2(a-1,n-2a-5,a-1))$ and $4a-n+3\geq a+2> \rho^2(G_2(a-1,n-2a-5,a-1))$. Thus, $4a-n+3-\rho^2>0$, implying $x_{v_1}>x_{v_2}$. It's easy to see that $G_2(a,n-2a-6,a-1)\cong G_2(a-1,n-2a-5,a-1)-v_2v_{a+5}+v_1v_{a+5}$. By Lemma \ref{lemPero}, we obtain that $\rho(G_2(a,n-2a-6,a-1))>\rho(G_2(a-1,n-2a-5,a-1))$, a contradiction.

For $n=3a+2$, let
\begin{equation*}
	\psi_{1}(\mu)=f_2(\mu,3a+2)-f_3(\mu,3a+2)=(\mu-1)(3\mu-3a+2a\mu-2\mu^2+3)
\end{equation*}

Solving equation $\psi_{1}(\mu)=0$, we obtain $\tilde{\mu}_{1,1}=\frac{a}{2}+\frac{\sqrt{4a^2-12a+33}}{4}+\frac{3}{4}$, $\tilde{\mu}_{1,2}=\frac{a}{2}-\frac{\sqrt{4a^2-12a+33}}{4}+\frac{3}{4}$, and $\tilde{\mu}_{1,3}=1$.

Notice that $\mu_{2,1}(3a+2)=\rho^2(G_2(a,a-4,a-1))>a+1> \tilde{\mu}_{1,1}=\frac{a}{2}+\frac{\sqrt{4a^2-12a+33}}{4}+\frac{3}{4}$ for $a\geq3$. 
Then we have
\begin{equation*}
		\psi_{1}(\mu_{2,1}(3a+2))=\psi_{1}\left( \rho^2\left( G_2(a,a-4,a-1)\right) \right)<\psi_{1}\left( a+1 \right)=2a(2-a)<0,
\end{equation*}
which implies that $f_3(\mu_{2,1}(3a+2),3a+2)=f_3(\rho^2\left( G_2(a,a-4,a-1)\right),3a+2)>0$. Since $\mu_{2,1}(3a+2)=\rho^2\left( G_2(a,a-4,a-1)\right)>a+1>\mu_{3,2}=\frac{a}{2}+\frac{\sqrt{a^2-2a+5}}{2}+\frac{1}{2}$, we conclude that $\mu_{2,1}(3a+2)>\mu_{3,1}(3a+2)$. Hence, $\rho(G_2(a,a-4,a-1))>\rho(G_2(a-1,a-3,a-1))$.

For $n=3a+3$, let
\begin{equation*}
	\psi_{2}(\mu)=f_2(\mu,3a+3)-f_3(\mu,3a+3)=(\mu-1)(2\mu-2a+a\mu-\mu^2+2)
\end{equation*}

Solving equation $\psi_{2}(\mu)=0$, we get $\tilde{\mu}_{2,1}=\frac{a}{2}+\frac{\sqrt{a^2-4a+12}}{2}+1$, $\tilde{\mu}_{2,2}=\frac{a}{2}-\frac{\sqrt{a^2-4a+12}}{2}+1$, and $\tilde{\mu}_{2,3}=1$.

Note that $\mu_{2,1}(3a+3)=\rho^2(G_2(a,a-3,a-1))>a+1>\tilde{\mu}_{2,1}=\frac{a}{2}+\frac{\sqrt{a^2-4a+12}}{2}+1$ for $a\geq3$. Then we obtain that 
\begin{equation*}
	\psi_{2}(\mu_{2,1}(3a+3))=\psi_{2}\left( \rho^2\left( G_2(a,a-3,a-1)\right) \right)<\psi_{2}(a+1)=a(3-a)\leq 0,
\end{equation*}
which implies that $f_{3}(\mu_{2,1}(3a+3),3a+3)=f_{3}(\rho^2(G_2(a,a-3,a-1)),3a+3)>0$. Since $\mu_{2,1}(3a+3)=\rho^2(G_2(a,a-3,a-1))>a+1>\mu_{3,2}=\frac{a}{2}+\frac{\sqrt{a^2-2a+5}}{2}+\frac{1}{2}$, one can see that $\mu_{2,1}(3a+3)>\mu_{3,1}(3a+3)$. Thus, $\rho(G_2(a,a-3,a-1))>\rho(G_2(a-1,a-2,a-1))$.

For $n=3a+4$, let
\begin{equation*}
	\psi_{3}(\mu)=f_2(\mu,3a+4)-f_1(\mu,3a+4)=(\mu-1)(2a-3\mu-a\mu+\mu^2)
\end{equation*}

Solving equation $\psi_{3}(\mu)=0$, we get $\tilde{\mu}_{3,1}=\frac{a}{2}+\frac{\sqrt{a^2-2a+9}}{2}+\frac{3}{2}$, $\tilde{\mu}_{3,2}=\frac{a}{2}-\frac{\sqrt{a^2-2a+9}}{2}+\frac{3}{2}$, and $\tilde{\mu}_{3,3}=1$.

By plugging the value $\tilde{\mu}_{3,1}$ into $\mu$ of $f_1(\mu,3a+4)$,
\begin{equation*}
	\begin{split}
		f_1(\tilde{\mu}_{3,1},3a+4)&=\left( 3-2a+\frac{a^2}{2}\right)\cdot\sqrt{a^2-2a+9}-\left( \frac{a^3}{2}-\frac{5}{2}a^2+7a-9\right)\\
		&=\frac{2a^2}{\left( 3-2a+\frac{a^2}{2}\right)\cdot\sqrt{a^2-2a+9}+\left( \frac{a^3}{2}-\frac{5}{2}a^2+7a-9\right)}\\
		&>0.
	\end{split}  
\end{equation*}

Moreover, we calculate that
\begin{equation*}
	\begin{split}
		\tilde{\mu}_{3,1}-\mu_{1,2}&=\frac{1}{2}\left[\left( \sqrt{a^2-2a+9}+1\right) -\sqrt{a^2+4} \right]\\
		&= \frac{\sqrt{a^2-2a+9}-(a-3)}{\left( \sqrt{a^2-2a+9}+1\right) +\sqrt{a^2+4}}\\
		&=\frac{4a}{\left[ \left( \sqrt{a^2-2a+9}+1\right) +\sqrt{a^2+4}\right] \left[ \sqrt{a^2-2a+9}+(a-3)\right] }\\
		&>0.
	\end{split}
\end{equation*}

Then we have
\begin{equation*}
	\tilde{\mu}_{3,2}<\mu_{1,2}<\mu_{1,1}(3a+4)=\rho^2(G_2(a,n-3,a))<\tilde{\mu}_{3,1},
\end{equation*}
which implies that
\begin{equation*}
	\psi_{3}(\mu_{1,1}(3a+4))=f_2(\mu_{1,1}(3a+4),3a+4)-f_1(\mu_{1,1}(3a+4),3a+4)<\psi_{3}(\tilde{\mu}_{3,1})=0.
\end{equation*}

Hence, $f_2(\mu_{1,1}(3a+4),3a+4)<0$, implying $\mu_{2,1}(3a+4)>\mu_{1,1}(3a+4)$, i.e., $\rho(G_2(a,a-2,a-1))>\rho(G_2(a,a-3,a))$.

For $n=3a+5$, let
\begin{equation*}
	\psi_{4}(\mu)=f_2(\mu,3a+5)-f_1(\mu,3a+5)=(a-\mu)(\mu-1).
\end{equation*}

Obviously, $\psi_{4}(\mu_{1,1}(3a+5))=f_2(\mu_{1,1}(3a+5),3a+5)-f_1(\mu_{1,1}(3a+5),3a+5)=(a-\mu_{1,1}(3a+5))(\mu_{1,1}(3a+5)-1)<0$.

Then we have $f_2(\mu_{1,1},3a+5)<0$, implying $\mu_{2,1}(3a+5)>\mu_{1,1}(3a+5)$. One can see that 
$\rho(G_2(a,a-1,a-1))>\rho(G_2(a,a-2,a))$.

As can be seen, the minimzer graph in $G_{n,\gamma}$ is isomorphic to some graph $G_2(a,n-2a-7,a)$, for $n\geq 11$ and $a\geq1$.
\end{proof}

\begin{theorem}
	For any graph $G\in \mathbb{G}_{n,3}$ with $n\geq 11$, then
	\begin{itemize}
		\item [\textnormal{(\romannumeral1)}] If $n=3m$, then $\rho(G)\geq\rho(G_2(m-2,m-3,m-2))$ and equality holds if and only if $G\cong G_2(m-2,m-3,m-2)$.
		\item [\textnormal{(\romannumeral2)}] If $n=3m+1$, then $\rho(G)\geq\rho(G_2(m-1,m-4,m-1))$ and equality holds if and only if $G\cong G_2(m-1,m-4,m-1)$.
		\item [\textnormal{(\romannumeral3)}] If $n=3m+2$, then $\rho(G)\geq\rho(G_2(m-1,m-3,m-1))$ and equality holds if and only if $G\cong G_2(m-1,m-3,m-1)$.
	\end{itemize}
\end{theorem}
\begin{proof}
	By the proof of Theorem 4.2, one can see that the characteristic polynomials of $G_{2}(a,n-2a-7,a)$ and $G_{2}(a-1,n-2a-5,a-1)$ are $f_1(\mu,n)$ and $f_{3}(\mu,n)$, respectively.
	
	Let
	\begin{equation*}
		\begin{split}
			\varphi(\mu,n)&=f_1(\mu,n)-f_3(\mu,n)\\
			&=-(\mu-1)\left( (6a-2n+9)\mu^2+(3n-18a+2an-6a^2-16)\mu+8a+n-2an+6a^2-5\right).
		\end{split}
	\end{equation*}

	Solving equation $\phi(\mu,n)=0$, we get
	\begin{equation*}
		\begin{split}
			\bar{\mu}_1(n) &=\frac{18a-3n-2an+6a^2+16}{2(6a-2n+9)} \\
			&~~~~+\frac{\sqrt{(4a^2-4a+17)n^2+(-24a^3-12a^2-60a-172)n+36a^4+72a^3+108a^2+408a+436}}{2(6a-2n+9)},\\
			\bar{\mu}_2(n) &= \frac{18a-3n-2an+6a^2+16}{2(6a-2n+9)} \\
			&~~~~-\frac{\sqrt{(4a^2-4a+17)n^2+(-24a^3-12a^2-60a-172)n+36a^4+72a^3+108a^2+408a+436}}{2(6a-2n+9)},
		\end{split}
	\end{equation*}
	and $\bar{\mu}_3(n)= 1$.
	
For $2a+7\leq n\leq 3a+2$ ($a\geq5$), one can see that $\bar{\mu}_{1}(n)>\bar{\mu}_{2}(n)$ and
\begin{equation*}
	\begin{split}
		&\bar{\mu}_{1}(n)\leq a+1\\
		&\Leftrightarrow \sqrt{(4a^2-4a+17)n^2+(-24a^3-12a^2-60a-172)n+36a^4+72a^3+108a^2+408a+436}\\
		&~~~~\leq 2(a+1)(6a-2n+9)-(18a-3n-2an+6a^2+16)\\
		&\Leftrightarrow \sqrt{(4a^2-4a+17)n^2+(-24a^3-12a^2-60a-172)n+36a^4+72a^3+108a^2+408a+436}\\
		&~~~~\leq 12a-(2a+1)n+6a^2+2\\
		&\Leftrightarrow (8a-16)n^2+(-48a^2 + 28a + 168)n + 72a^3 + 60a^2 - 360a - 432\geq0. 
	\end{split}
\end{equation*}


For $a\geq 5$, we have
\begin{equation*}
	\begin{split}
		&\frac{48a^2-28a-168}{2(8a-16)}>3a+2\\
		&\Leftrightarrow 36a-104>0\\
		&\Leftrightarrow a\geq 3.
	\end{split}
\end{equation*}

Thus, we get
\begin{equation*}
	\begin{split}
		&(8a-16)n^2+(-48a^2 + 28a + 168)n + 72a^3 + 60a^2 - 360a - 432\\
		&\geq(8a-16)(3a+2)^2+(-48a^2 + 28a + 168)(3a+2) + 72a^3 + 60a^2 - 360a - 432\\
		&=40a-160\geq 0.
	\end{split}
\end{equation*}

It's clear that $\varphi(\mu_{1,1}(n),n)<\varphi(a+1,n)\leq \varphi(\bar{\mu}_1,n)=0$. Hence, $f_3(\mu_{1,1}(n),n)>0$. Note that $\mu_{1,1}(n)>a+1>\mu_{3,2}=\frac{a}{2}+\frac{\sqrt{a^2-2a+5}}{2}+\frac{1}{2}$, we have $\mu_{1,1}(n)>\mu_{3,1}(n)$, implying $\rho(G_2(a,n-2a-7,a))>\rho(G_2(a-1,n-2a-5,a-1))$.

For $n= 3a+3$ ($a\geq4$), one can check that
\begin{equation*}
	\varphi(\mu,3a+3)=f_1(\mu,3a+3)-f_3(\mu,3a+3)=-3\mu^3+(3a+10)\mu^2+(-8a-5)\mu+5a-2.
\end{equation*}

Solving equation $\varphi(\mu,3a+3)=0$, we obtain
\begin{equation*}
	\begin{split}
		\bar{\mu}_{1}(3a+3)&=\frac{a}{2}+\frac{\sqrt{9a^2-18a+73}}{6}+\frac{7}{6},\\
		\bar{\mu}_{2}(3a+3)&=\frac{a}{2}-\frac{\sqrt{9a^2-18a+73}}{6}+\frac{7}{6},
	\end{split}
\end{equation*}
and
\begin{equation*}
	\bar{\mu}_{3}(3a+3)=1.
\end{equation*}

Consider the following inequality
\begin{equation*}
	\begin{split}
		&\bar{\mu}_{1}(3a+3)<\mu_{1,2}\\
		&\Leftrightarrow\frac{a}{2}+\frac{\sqrt{9a^2-18a+73}}{6}+\frac{7}{6}<\frac{a}{2}+\frac{\sqrt{a^2+4}}{2}+1\\
		&\Leftrightarrow \sqrt{9a^2-18a+73}<3\sqrt{a^2+4}-1\\
		&\Leftrightarrow \sqrt{a^2+4}<3(a-2)\\
		&\Leftrightarrow 2a^2-9a+8>0.
	\end{split}
\end{equation*}

It's easy to check that $2a^2-9a+8\geq4>0$. Then we have $\mu_{1,1}(3a+3)>\mu_{1,2}>\bar{\mu}_{1}(3a+3)$ and
\begin{equation*}
	\varphi(\mu_{1,1}(3a+3),3a+3)<\varphi(\bar{\mu}_{1}(3a+3),3a+3)=0.
\end{equation*}

Therefore, $f_3(\mu_{1,1}(3a+3),3a+3)>0$. Note that $\mu_{1,1}(3a+3)>a+1>\mu_{3,2}=\frac{a}{2}+\frac{\sqrt{a^2-2a+5}}{2}+\frac{1}{2}$, we have $\mu_{1,1}(3a+3)>\mu_{3,1}(3a+3)$, implying $\rho(G_2(a,a-4,a))>\rho(G_2(a-1,a-2,a-1))$.

For $n=3a+4$ ($a\geq3$), we get
\begin{equation*}
	\varphi(\mu,3a+4)=f_1(\mu,3a+4,)-f_3(\mu,3a+4)=-\mu^3+(a+5)\mu^2+(-4a-3)\mu+3a-1.
\end{equation*}

Solving equation $\varphi(\mu,3a+4)=0$, we have
\begin{equation*}
	\begin{split}
		\bar{\mu}_{1}(3a+4)&=\frac{a}{2}+\frac{\sqrt{a^2-4a+20}}{2}+2,\\
		\bar{\mu}_{2}(3a+4)&=\frac{a}{2}-\frac{\sqrt{a^2-4a+20}}{2}+2,
	\end{split}
\end{equation*}
and
\begin{equation*}
		\bar{\mu}_{3}(3a+4)=1.
\end{equation*}

By plugging the value $\bar{\mu}_{1}(3a+4)$ into $\mu$ of $f_1(\mu,3a+4)$, we get
\begin{equation*}
	f_1(\bar{\mu}_{1}(3a+4),3a+4)=\left( 2a^2-16a+\frac{87}{2}\right) \sqrt{a^2-4a+20}-\left( 2a^3-20a^2+\frac{183}{2}a-195\right) 
\end{equation*}

Note that $2a^3-20a^2+\frac{183}{2}a-195<0$ for $a=3,4$, $2a^3-20a^2+\frac{183}{2}a-195>0$ for $a\geq5$ and $2a^2-16a+\frac{87}{2}\geq11.5>0$. We only consider the case of $a\geq5$.

When $a\geq5$, one can see that
\begin{equation*}
	\begin{split}
		&\left( 2a^2-16a+\frac{87}{2}\right) \sqrt{a^2-4a+20}>\left( 2a^3-20a^2+\frac{183}{2}a-195\right)\\
		&\Leftrightarrow 48a^3 - 112a^2 + 276a - 180>0.
	\end{split}
\end{equation*}

Since $48a^3 - 112a^2 + 276a - 180\geq4400>0$ for $a\geq5$, we have $f_1(\bar{\mu}_{1}(3a+4),3a+4)>0$. One can check that $\bar{\mu}_{1}(3a+4)=\frac{a}{2}+\frac{\sqrt{a^2-4a+20}}{2}+2>\mu_{1,2}=\frac{a}{2}+\frac{\sqrt{a^2+4}}{2}+1>\bar{\mu}_{2}(3a+4)=\frac{a}{2}-\frac{\sqrt{a^2-4a+20}}{2}+2$. Thus, we obtain $\bar{\mu}_{1}(3a+4)>\mu_{1,1}(3a+4)>\mu_{1,2}>\bar{\mu}_{2}(3a+4)$ for $a\geq3$.

Hence, $\varphi(\mu_{1,1}(3a+4),3a+4)>0$. It can be checked that $f_3(\mu_{1,1}(3a+4),3a+4)<0$, implying $\mu_{1,1}(3a+4)<\mu_{3,1}(3a+4)$ for $a\geq3$. Hence, $\rho(G_2(a-1,a-1,a-1))>\rho(G_2(a,a-3,a))$.

For $n\geq 3a+5$ ($a\geq2$), we calculate that $6a-2n+9<0$, $\bar{\mu}_{2}(n)>\bar{\mu}_{1}(n)$, and
\begin{equation*}
	\begin{split}
		&\bar{\mu}_{2}(n)< a+1\\
		&\Leftrightarrow \sqrt{(4a^2-4a+17)n^2+(-24a^3-12a^2-60a-172)n+36a^4+72a^3+108a^2+408a+436}\\
		&~~~~<(18a-3n-2an+6a^2+16)-2(a+1)(6a-2n+9)\\
		&\Leftrightarrow \sqrt{(4a^2-4a+17)n^2+(-24a^3-12a^2-60a-172)n+36a^4+72a^3+108a^2+408a+436}\\
		&~~~~<(2a+1)n-6a^2-12a+2\\
		&\Leftrightarrow (8a-16)n^2+(-48a^2 + 28a + 168)n + 72a^3 + 60a^2 - 360a - 432>0.
	\end{split}
\end{equation*}

Since
\begin{equation*}
	\frac{48a^2-28a-168}{2(8a-16)}<3a+5\Leftrightarrow -12a-8<0,
\end{equation*}
we obtain
\begin{equation*}
	\begin{split}
		&~~~~(8a-16)n^2+(-48a^2 + 28a + 168)n + 72a^3 + 60a^2 - 360a - 432\\
		&\geq(8a-16)\cdot(3a+5)^2+(-48a^2 + 28a + 168)\cdot(3a+5) + 72a^3 + 60a^2 - 360a - 432\\
		&=4a+8>0.
	\end{split}
\end{equation*}

Then we have $\mu_{1,1}(n)>a+1>\bar{\mu}_{2}(n)$, implying $\varphi(\mu_{1,1}(n),n)>\varphi(\bar{\mu}_{2}(n),n)=0$. Thus, $f_{3}(\mu_{1,1}(n),n)<0$. It can be conclude that $\mu_{3,1}(n)>\mu_{1,1}(n)$ for $n\geq 3a+5$, i.e., $\rho(G_2(a-1,n-2a-5,a-1))>\rho(G_2(a,n-2a-7,a))$.

All in all, one can see that
\begin{equation*}
	\begin{cases}
		\rho(G_2(a,n-2a-7,a))>\rho(G_2(a-1,n-2a-5,a-1)), & \mbox{for } 2a+7\leq n\leq 3a+3, a\geq4, \\
		\rho(G_2(a,a-3,a))<\rho(G_2(a-1,a-1,a-1)), & \mbox{for } n=3a+4, a\geq3, \\
		\rho(G_2(a,n-2a-7,a))<\rho(G_2(a-1,n-2a-5,a-1)), & \mbox{for } n\geq3a+5, a\geq2.
	\end{cases}
\end{equation*}

For $n\geq 16$, we have
\begin{equation*}
	\begin{cases}
		\rho(G_2(a,n-2a-7,a))>\rho(G_2(a-1,n-2a-5,a-1)), & \mbox{for } \frac{n-3}{3}<a<\frac{n-7}{2},\\
		\rho(G_2(a,n-2a-7,a))<\rho(G_2(a-1,n-2a-5,a-1)), & \mbox{for } 4\leq a\leq\frac{n-4}{3}.
	\end{cases}
\end{equation*}

For $11\leq n\leq 15$, we have
\begin{equation*}
	\rho(G_2(a,n-2a-7,a))<\rho(G_2(a-1,n-2a-5,a-1)), \mbox{ for }  2\leq a\leq \dfrac{n-5}{3}<4.
\end{equation*}

For $n=3m$ ($m\geq4$), it follows that
\begin{equation*}
	\begin{split}
		&\rho\left( G_2\left( \lfloor\frac{3}{2}m-\frac{7}{2}\rfloor,3m-2\lfloor\frac{3}{2}m-\frac{7}{2}\rfloor-7,\lfloor\frac{3}{2}m-\frac{7}{2}\rfloor\right)  \right)>\dots>\rho(G_2(m-3,m-1,m-3))\\
		&>\rho(G_2(m-2,m-3,m-2))<\rho(G_2(m-1,m-5,m-1))<\dots<\rho(G_2(2,3m-11,2)).
	\end{split}
\end{equation*}

For $n=3m+1$ ($m\geq4$), it follows that
\begin{equation*}
	\begin{split}
		&\rho\left( G_2\left( \lfloor\frac{3}{2}m-3\rfloor,3m-2\lfloor\frac{3}{2}m-3\rfloor-6,\lfloor\frac{3}{2}m-3\rfloor\right)  \right)>\dots>\rho(G_2(m,m-6,m))\\
		&>\rho(G_2(m-1,m-4,m-1))<\rho(G_2(m-2,m-2,m-2))<\dots<\rho(G_2(2,3m-10,2)).
	\end{split}
\end{equation*}

For $n=3m+2$ ($m\geq3$), it follows that
\begin{equation*}
	\begin{split}
		&\rho\left( G_2\left( \lfloor\frac{3}{2}m-\frac{5}{2}\rfloor,3m-2\lfloor\frac{3}{2}m-\frac{5}{2}\rfloor-5,\lfloor\frac{3}{2}m-\frac{5}{2}\rfloor\right)  \right)>\dots>\rho(G_2(m,m-5,m))\\
		&>\rho(G_2(m-1,m-3,m-1))<\rho(G_2(m-2,m-1,m-2))<\dots<\rho(G_2(2,3m-9,2)).
	\end{split}
\end{equation*}

This completes the proof.
\end{proof}

\section{The minimizer graphs in $\mathbb{G}_{n,\lfloor \frac{n}{2}\rfloor}$}
In this section, we determine graphs with the minimum spectral radius in $\mathbb{G}_{n,\lfloor \frac{n}{2}\rfloor}$.

\begin{definition}
	The corona of two graph $H_1$ and $H_2$ is the graph $H=H_1\circ H_2$ formed from one copy of $H_1$ and $|V(H_1)|$ copies of $H_2$ where the $i$-th vertex of $H_1$ is adjacent to every vertex in the $i$-th copy of $H_2$.
\end{definition}

\begin{example}
	The corona $H\circ K_1$ is the graph constructed from a copy of $H$, where for each vertex $v\in V(H)$, a new vertex $v'$ and a pendent $vv'$ are added.
\end{example}

\begin{figure}[htbp]
	\centering 
	\includegraphics[width=7.5cm]{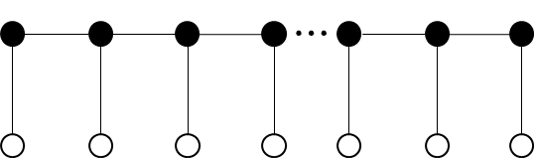}
	\caption{The graph $P_{\frac{n}{2}}\circ K_1$.}
\end{figure}

\begin{lemma}\cite{Finkdomin,Paydomin}\label{lemcirc}
	An $n$-vertex tree $T$ with domination number $\gamma$ satisfies $\gamma=\frac{n}{2}$ (for $n$ even) if and only if there exists a $\gamma$-vertex tree $H$ $(\gamma=\frac{n}{2})$ such that $T\cong H\circ K_1$. 
\end{lemma}

\begin{theorem}
	Let $n\geq2$ be even. For any graph $G\in\mathbb{G}_{n,\lfloor\frac{n}{2}\rfloor}$, then $\rho(G)\geq \rho(P_{\frac{n}{2}}\circ K_1)$ and equality holds if and only if $G\cong P_{\frac{n}{2}}\circ K_1$.
\end{theorem}
\begin{proof}
	By Theorem \ref{themtree}, the minimizer graph in $\mathbb{G}_{n,\frac{n}{2}}$ must be a tree.
	From Lemma \ref{lemcirc}, a tree $T$ with $\gamma=\frac{n}{2}$ if and only if $T\cong H\circ K_1$ for some tree $H$ on $\frac{n}{2}$ vertices. For any such $H$, $A(H\circ K_1)$ has the form
	\begin{equation*}
		\begin{bmatrix}
			A(H) & I_{\frac{n}{2}}\\
			I_{\frac{n}{2}} & 0
		\end{bmatrix},
	\end{equation*} 
	where $I_{\frac{n}{2}}$ is an identity matrix of order $\frac{n}{2}$.
	
	\begin{equation*}
		\begin{split}
			\Phi(H\circ K_1,\lambda)&=|\lambda I_{n}-A(H\circ K_1)|\\
			&=\left|\begin{matrix}
				\lambda I_{\frac{n}{2}}-A(H), & -I_{\frac{n}{2}} \\
				-I_{\frac{n}{2}} & \lambda I_{\frac{n}{2}}
			\end{matrix}\right|\\
		&=\left|\begin{matrix}
			\left(\lambda-\frac{1}{\lambda} \right) I_{\frac{n}{2}}-A(H) & -I_{\frac{n}{2}}\\
			0 & \lambda I_{\frac{n}{2}}
		\end{matrix}\right|\\
	&=\lambda|I_{\frac{n}{2}}|\cdot\left|\left( \lambda-\frac{1}{\lambda}\right)I_{\frac{n}{2}}-A(H) \right|\\
	&=\lambda\left|\left( \lambda-\frac{1}{\lambda}\right)I_{\frac{n}{2}}-A(H) \right|\\
	&=\lambda\Phi\left( H,\lambda-\frac{1}{\lambda}\right) 
		\end{split}
	\end{equation*}

	Let $\sigma=\lambda-\frac{1}{\lambda}$, then we have $\Phi(H\circ K_1,\lambda)=\lambda\Phi(H,\sigma)$ and $\lambda=\frac{\sigma\pm\sqrt{\sigma^2+4}}{2}$. 
	
	Let $\lambda_{1}$ and $\sigma_{1}$ be the largest roots of polynomials $\Phi(H\circ K_1,\lambda)$ and $\Phi(H,\sigma)$, respectively. Note that $H$ is a proper graph of $H\circ K_1$. It follows that $\lambda_1>\sigma_{1}$. It can be conclude that $\lambda_{1}=\frac{\sigma_1+\sqrt{\sigma_{1}^{2}+4}}{2}$. Therefore, the spectral radius of $H\circ K_1$ is increasing in the spectral radius of $H$. Combining Lemma \ref{lemmintree}, we have $G\cong P_{\frac{n}{2}}\circ K_1$ minimizers the spectral radius over graphs on $n$ vertices domination number $\gamma$.
\end{proof}
	
	If $n$ is odd, let $H'$ be a tree obtained from $P_{\lfloor\frac{n}{2}\rfloor}\circ K_1$ by once subdivision of a pendent edge of the diameter of $P_{\lfloor\frac{n}{2}\rfloor}\circ K_1$. The following conclusion does not seem so easy to prove.
\begin{conjecture}
	Let $n\geq3$ be odd. For any graph $G\in\mathbb{G}_{n,\lfloor\frac{n}{2}\rfloor}$, then $\rho(G)\geq \rho(H)$ and equality holds if and only if $G\cong H'$.
\end{conjecture}
\begin{figure}[htbp]
	\centering 
	\includegraphics[width=8cm]{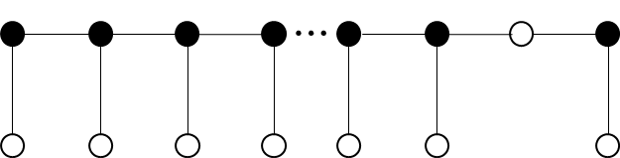}
	\caption{The graph $H'$.}
	\label{conject}
\end{figure}

\section*{Declaration of competing interest}
The authors declare that they have no known competing financial interests or personal relationships that could have appeared to influence the work reported in this paper.

\section*{Acknowledgement(s)}
This work was supported by the National Natural Science Foundation
of China (Grant No. 61773020). The authors would like to express their sincere gratitude to all the referees for their careful reading and insightful suggestions.

\end{document}